\documentclass[12pt, reqno]{amsart}
\usepackage{amsmath,amsthm,amscd,amsfonts,color,mathrsfs}
\usepackage{amssymb, mathtools, lscape, float}
\usepackage{comment}
\usepackage{enumerate}
\usepackage{graphicx} 
\usepackage[shortlabels]{enumitem}
\usepackage{enumerate}
\usepackage{MnSymbol}
\usepackage[bookmarksnumbered, colorlinks, plainpages]{hyperref}
\hypersetup{colorlinks=true,linkcolor=blue, anchorcolor=green, citecolor=blue, urlcolor=red, filecolor=magenta, pdftoolbar=true}

\newtheorem{theorem}{Theorem}[section]

\newtheorem{Result}[theorem]{Result}
\newtheorem{proposition}[theorem]{Proposition}
\newtheorem{question}[theorem]{Question}
\newtheorem{remark}[theorem]{Remark}
\newtheorem{example}[theorem]{Example}
\newtheorem{lemma}[theorem]{Lemma}
\newtheorem{corollary}[theorem]{Corollary}
\newcommand\restr[2]{{
  \left.\kern-\nulldelimiterspace 
  #1 
  \littletaller 
  \right|_{#2} 
  }}

\newcommand{\littletaller}{\mathchoice{\vphantom{\big|}}{}{}{}}

\date{}
\begin{document}

\title [Weighted composition operators]{Numerical range and invariant subspaces of weighted composition operators on the Hardy space of Dirichlet series}

\author[S. Lahiri, S. Ojha, S. Halder, R. Birbonshi]{Sudeshna Lahiri, Sarita Ojha, Subhadip Halder, Riddhick Birbonshi}
	
\address[Lahiri]{Department of Mathematics, Indian Institute of Engineering Science and Technology, Shibpur, Howrah 711103, West Bengal, India}
\email{sudeshna.lahiri88@gmail.com}
	
\address[Ojha]{Department of Mathematics, Indian Institute of Engineering Science and Technology, Shibpur, Howrah 711103, West Bengal, India}
\email{sarita.ojha89@gmail.com}

\address[Halder] {Department of Mathematics, Jadavpur University, Kolkata 700032, West Bengal, India}
\email{subhadiphalderju@gmail.com}

\address[Birbonshi] {Department of Mathematics, Jadavpur University, Kolkata 700032, West Bengal, India}
\email{riddhick.math@gmail.com}



\begin{abstract}
In this paper, we study various results on numerical range of weighted composition operators on Hardy space of Dirichlet series. Certain conditions are provided under which the numerical range contains zero and some circular or elliptic disks. With the help of invariant subspaces of the operator, we also determine the numerical range of reductive weighted composition operators on this space. Additionally, we discuss some results on the Davis-Wielandt shell of composition operators with some natural questions that arise from our findings.
\end{abstract}

\maketitle

\tableofcontents

\section{Introduction}
The space $\mathcal{H}^2$ denotes the Hardy space of Dirichlet series with square summable coefficients, i.e.,
$$\mathcal{H}^2=\left\{f(s)=\sum_{n=1}^\infty a_n n^{-s}:\|f\|=\left(\sum_{n=1}^\infty|a_n|^2\right)^{1/2}<\infty\right\}.$$
By the Cauchy-Schwartz inequality, the functions in $\mathcal{H}^2$ are all analytic on the half-plane $\mathbb{C}_{1/2}=\{s\in \mathbb{C}: \Re s>1/2\}$. For $f(s)=\sum_{n=1}^\infty a_n n^{-s}$ and $g(s)=\sum_{n=1}^\infty b_n n^{-s}$ in $\mathcal{H}^2$, the inner product on $\mathcal{H}^2$ is given by
$$\langle f,g\rangle=\sum_{n=1}^\infty a_n \overline{b_n}\ .$$
Thus $\mathcal{H}^2$ appears as a Hilbert space of analytic functions on $\mathbb{C}_{1/2}$ where the reproducing kernel at $\omega\in\mathbb{C}_{1/2}$ is given by $K_\omega(s)=\zeta(\overline{\omega}+s)$ for $\Re s>1/2$ and $\|K_\omega\|=(\zeta(2\Re \omega))^{1/2}$ where $\zeta$ is the Riemann zeta function $\zeta(z)=\displaystyle\sum_{n=1}^\infty n^{-z}$ for $\Re z>1$. It follows that $\langle f, K_\omega\rangle=f(\omega)$ for any $f$ in $\mathcal{H}^2$. For $n\geq 1$, the functions $e_n(s)=n^{-s}\in \mathcal{H}^2$ form an orthonormal basis of $\mathcal{H}^2$. Let $\mathcal{D}$ denote the space of functions which can be represented by a convergent Dirichlet series $$f(s)=\sum_{n=1}^\infty a_n n^{-s}; \ a_n\in \mathbb{C}$$ in some half-plane $\sigma_c(f)<\Re s<+\infty$ where $\sigma_c(f)$ is defined by
\begin{align*}
    \sigma_c(f)=\inf \{\sigma\in\mathbb{R}:\sum_{n=1}^\infty a_n n^{-\sigma}\ \text{converges}\}.
\end{align*}
Let $\mathbb{C}_0$ denote the right half-plane $\{s\in \mathbb{C}: \Re s>0\}$. The space $H^\infty(\mathbb{C}_0)$ denotes the Banach space of bounded and analytic functions on $\mathbb{C}_0$ endowed with the norm
$$\|f\|_{\infty}=\sup_{\Re s>0} |f(s)|.$$
Hedenmalm et al. \cite[Theorem 3.1]{hedenmalm1997hilbert} have shown that the set of multipliers of $\mathcal{H}^2$ denoted by $\mathcal{M}$ is the subset of $H^\infty(\mathbb{C}_0)$ belonging to $\mathcal{D}$. For an analytic function $\psi$ on $\mathbb{C}_{1/2}$ and an analytic self-map $\phi$ on $\mathbb{C}_{1/2}$, the weighted composition operator on $\mathcal{H}^2$ with symbols $\psi$ and $\phi$ is defined by
$$C_{\psi,\phi} f=\psi\cdot (f\circ \phi)$$ for all $f$ in $\mathcal{H}^2$. In particular, when $\psi\equiv1$, $C_{\psi,\phi}$ becomes the composition operator $C_\phi$ and when $\phi$ is the identity map on $\mathbb{D}$, $C_{\psi,\phi}$ becomes the multiplication operator. Whenever $C_\phi$ maps $\mathcal{H}^2$ into $\mathcal{H}^2$, the closed graph theorem implies that $C_\phi$ is bounded (see \cite{gordon1999composition}). More explicitly, Gordon and Hedenmalm \cite[Theorem B]{gordon1999composition} have provided the following necessary and sufficient condition for boundedness of composition operator on $\mathcal{H}^2$. 
\begin{theorem}\cite{gordon1999composition}\label{bdd}
    An analytic function $\phi:\mathbb{C}_{1/2}\to\mathbb{C}_{1/2}$ defines a bounded composition operator $C_\phi:\mathcal{H}^2\to\mathcal{H}^2$ if and only if
    \begin{enumerate}[(a)]
        \item it is of the form 
        $$\phi(s)=c_0s+\eta(s)$$
        where $c_0\in\mathbb{N}\cup\{0\}$ and $\eta(s)=\sum_{n=1}^\infty c_n n^{-s}\in\mathcal{D}$, and
        \item $\phi$ has an analytic extension to $\mathbb{C}_0$, also denoted by $\phi$ such that
        \begin{enumerate}[{(i)}]
            \item $\phi(\mathbb{C}_0)\subset \mathbb{C}_0$ if $c_0>0$, and
            \item $\phi(\mathbb{C}_0)\subset \mathbb{C}_{1/2}$ if $c_0=0$.
        \end{enumerate}
    \end{enumerate}
\end{theorem}
We say that $\phi$ is a $c_0$-symbol if it satisfies the conditions of Theorem \ref{bdd}. If $\phi$ is a $c_0$-symbol and $\psi\in \mathcal{M}$, then $C_{\psi,\phi}$ is bounded on $\mathcal{H}^2$. On the other hand, if $C_{\psi,\phi}$ is bounded on $\mathcal{H}^2$, then $\psi\in \mathcal{H}^2$. 
For any $\omega\in \mathbb{C}_{1/2}$, the adjoint of $C_{\psi,\phi}$ satisfies the relation
\begin{equation*}
C_{\psi,\phi}^*K_\omega=\overline{\psi(\omega)}K_{\phi(\omega)}.
\end{equation*}
Composition operators and weighted composition operators on the Hardy space of Dirichlet series have been extensively studied in the existing literature (see \cite{bayart2002hardy,gordon1999composition,maofa2018weighted, yao2021complex} and the references therein).
  \par Let $H$ be a complex separable Hilbert space and $\mathbb{B}(H)$ denote the algebra of all bounded linear operators on $H$. The numerical range $W(T)$ and the numerical radius $w(T)$ of an operator $T \in \mathbb{B}(H)$ is defined as $$ W(T)= \{\langle Tx , x \rangle  : x \in H ,\ \| x \| =1 \}\ \text{and}\ w(T)=\{\sup |\lambda|:\lambda \in W(T)\}.$$
It is well known that $W(T)$ is a bounded and convex subset of $\mathbb{C}$. Also, the spectrum of $T$ is contained in the closure of $W(T)$. Further details in this topic can be found in \cite{gau2021numerical,gustafson1997numerical}.
\par The question of whether $0$ lies in the numerical range of composition operators on the classical Hardy space was first investigated by Bourdon and Shapiro \cite{bourdon2002zero} and subsequently studied for weighted composition operators on various function spaces \cite{gunatillake2014numerical,halder2026note,sen2025numerical,shaabani2022numericalfock}.
   Finet and Queffélec \cite{finet2004numerical} have studied the numerical range of composition operators on Hardy space of Dirichlet series. It is therefore interesting to focus on the numerical range of weighted composition operators on the Hardy space of Dirichlet series. 
   \par Moreover, Wang and Yao \cite{wang2015invariant} have examined the invariant subspaces of composition operators on this space. Extending their results, we study invariant subspaces of weighted composition operators and, as a consequence, determine the numerical range of reductive weighted composition operators on the Hardy space of Dirichlet series. 
   \par Furthermore, the Davis-Wielandt shell has been widely studied as a generalization of the numerical range. Recently, Liu and Liu \cite{liu2026davis} have investigated the Davis-Wielandt shell of composition operators on the Hardy space. Accordingly, the study of Davis-Wielandt shell of composition operators on the Hardy space of Dirichlet series emerges as a natural and significant problem.\\
The present paper is organized as follows:\par
 In Section \ref{Sec2}, we discuss the results concerning the inclusion of zero in the numerical range of $C_{\psi,\phi}$ on $\mathcal{H}^2$. In Section \ref{Sec3}, we provide sufficient conditions under which certain circular or elliptic disks are contained in the numerical range of the operator. Section \ref{Sec4} deals with the invariant and reducing subspaces of this operator. Finally, in Section \ref{Sec5}, we give some additional observations on the Davis-Wielandt shell of composition operators on $\mathcal{H}^2$ and discuss some questions emerging from our analysis.

\section{Inclusion of zero in the numerical range}\label{Sec2}

In this section, we give specific conditions under which the numerical range of a weighted composition operator defined on Hardy space of Dirichlet series contains zero, and determine whether $0$ lies in its interior. In \cite[Proposition 8]{finet2004numerical}, Finet and Queffélec have shown that for any composition operator with non-identity symbol, $0$ belongs to the closure of its numerical range. The following result extends this for weighted composition operators. 

\begin{theorem}\label{sum}
    Let $\phi_i\ (1\leq i\leq n)$ be $c_0$-symbols and $\psi_i\ (1\leq i\leq n)$ be analytic functions on the half-plane $\mathbb{C}_{1/2}$ such that $C_{\psi_i,\phi_i}\in \mathbb{B}(\mathcal{H}^2)$ for all $i$ with $1\leq i\leq n$.
    \begin{enumerate}[label=(\roman*)]
    \item \label{i} If $\psi_i$'s have a common zero in $\mathbb{C}_{1/2}$ for all $i$ with $1\leq i\leq n$, then $0\in W\left(\displaystyle \sum_{i=1}^n C_{\psi_i,\phi_i}\right).$
    \item \label{ii} If, for all $i$ with $1\leq i\leq n$, $\phi_i$'s are non-identity maps on $\mathbb{C}_{1/2}$ and $\psi_i\in H^\infty(\mathbb{C}_0)$, then $0\in \overline{W\left(\displaystyle \sum_{i=1}^n C_{\psi_i,\phi_i}\right)}.$
    \end{enumerate}
\end{theorem}
\begin{proof}
  For any $\omega\in \mathbb{C}_{1/2}$, we have,
    \begin{align*}
        \left\langle \left(\displaystyle \sum_{i=1}^n C_{\psi_i,\phi_i}\right) \frac{K_\omega}{\|K_\omega\|},\frac{K_\omega}{\|K_\omega\|}\right\rangle&=\left\langle\frac{K_\omega}{\|K_\omega\|}, \left(\displaystyle \sum_{i=1}^n C_{\psi_i,\phi_i}\right)^* \frac{K_\omega}{\|K_\omega\|}\right\rangle\\
        &=\left\langle\frac{K_\omega}{\|K_\omega\|},\frac{1}{\|K_\omega\|}\sum_{i=1}^n\overline{\psi_i(\omega)}K_{\phi_i(\omega)}\right\rangle.
    \end{align*}
    \begin{enumerate}[label=(\roman*)]
    \item
    Let $\psi_i$'s have a common zero (say, $\omega_0$) in $\mathbb{C}_{1/2}$, then
    $ \left\langle \left(\displaystyle \sum_{i=1}^n C_{\psi_i,\phi_i}\right) \frac{K_{\omega_0}}{\|K_{\omega_0}\|},\frac{K_{\omega_0}}{\|K_{\omega_0}\|}\right\rangle=0.$ Thus $0\in W\left(\displaystyle \sum_{i=1}^n C_{\psi_i,\phi_i}\right).$
    \item 
    Let $L_{1/2}=\{s\in \mathbb{C}: \Re s=1/2\}$ and $A_i=\{s\in L_{1/2}:\phi_i(s)\neq s\}$ for $1\leq i \leq n$. Since $\phi_i:\mathbb{C}_{1/2}\to \mathbb{C}_{1/2}$ is not identity for $1\leq i \leq n$, the sets $A_i^c$ are countable sets. Thus it follows that, $\displaystyle\bigcap_{i=1}^nA_i$ is non-empty, i.e., there exists $a\in L_{1/2}$ such that $\phi_i(a)\neq a$ for $1\leq i \leq n$. Let $\epsilon>0$ and set $a_\epsilon =a+\epsilon$. Then $w_\epsilon=\left\langle \left(\displaystyle \sum_{i=1}^n C_{\psi_i,\phi_i}\right) \frac{K_{a_\epsilon}}{\|K_{a_\epsilon}\|},\frac{K_{a_\epsilon}}{\|K_{a_\epsilon}\|}\right\rangle\in W\left(\displaystyle \sum_{i=1}^n C_{\psi_i,\phi_i}\right)$. Now,
    \begin{align*}
        w_\epsilon &=\frac{1}{\|K_{a_\epsilon}\|^2}\langle K_{a_\epsilon},\sum_{i=1}^n\overline{\psi_i(a_\epsilon)}K_{\phi_i(a_\epsilon)}\rangle\\
        &=\frac{1}{\|K_{a_\epsilon}\|^2}\sum_{i=1}^n\psi_i(a_\epsilon)\langle K_{a_\epsilon},K_{\phi_i({a_\epsilon})}\rangle\\
        &=\frac{1}{\|K_{a_\epsilon}\|^2}\sum_{i=1}^n\psi_i(a_\epsilon)\zeta(\overline{a_\epsilon}+\phi_i(a_\epsilon)):=\frac{B_\epsilon}{D_\epsilon}
    \end{align*}
    Also $\psi_i\in H^\infty(\mathbb{C}_0)$ and since $\phi_i(a)\neq a$, therefore we have, $\phi_i(a)+\overline{a}=\phi_i(a)+1-a\neq 1$. Hence $B_\epsilon$ tends to the finite value $\sum_{i=1}^n\psi_i(a)\zeta(\overline{a}+\phi_i(a))$ as $\epsilon\to 0$. Since the only singularity of the meromorphic continuation of $\zeta(s)$ into $\mathbb{C}$ is a simple pole at $s=1$, therefore,  $D_\epsilon\to\infty$ as $\epsilon\to 0$. Thus $w_\epsilon\to 0$ and hence $0\in \overline{W\left(\displaystyle \sum_{i=1}^n C_{\psi_i,\phi_i}\right)}$.
\end{enumerate}
\end{proof}

\noindent The following result is obtained from Theorem \ref{sum} by taking $n=1$. 
\begin{corollary}\label{closure}
    Let $\phi$ be a $c_0$-symbol which is not an identity on $\mathbb{C}_{1/2}$ and $\psi\in H^\infty(\mathbb{C}_0)$ such that $C_{\psi,\phi}\in \mathbb{B}(\mathcal{H}^2)$. Then $0\in \overline{W\left(\displaystyle C_{\psi,\phi}\right)}.$
\end{corollary}


\noindent The next proposition provides the numerical range of $C_{\psi,\phi}$ when $\phi$ is assumed to be constant. The following results are required to prove our main result.
\begin{proposition}\label{constant}
    For $\phi\equiv c_1(\in\mathbb{C}_{1/2})$ and $\psi\in \mathcal{H}^2$, the following statements hold:
    \begin{enumerate}[label=(\roman*)]
        \item If $K_{c_1}=\mu\psi$ for some $\mu\neq 0$, then $W(C_{\psi,\phi})=[0,\overline{\mu}\|\psi\|^2]$.
        \item If $K_{c_1}\perp \psi$, then $W(C_{\psi,\phi})$ is the closed disk centered at origin and radius $\displaystyle\frac{\|\psi\| (\zeta(2\Re c_1))^{1/2} }{2}$.
        \item Otherwise, $W(C_{\psi,\phi})$ is the closed elliptic disk with foci at $0$ and $\psi(c_1)$.
        
    \end{enumerate}
\end{proposition}

\begin{proof}
   For any $f\in \mathcal{H}^2$, we have
    $$C_{\psi,\phi}f=\psi\cdot f(c_1)=\langle f,K_{c_1}\rangle\psi.$$
    Thus, $C_{\psi,\phi}$ is an operator with rank one. Hence the result follows from \cite[Proposition 2.5]{bourdon2002zero}.
\end{proof}

\noindent The case when $\phi$ is non-constant $c_0$-symbol with $c_0=0$ is considered in the next result.

\begin{proposition}\label{c0=0}
    Suppose that $\phi$ is a non-constant $c_0$-symbol and $\psi\in \mathcal{H}^2$ such that $C_{\psi,\phi}$ is bounded. If $c_0=0$, then $0$ lies in the interior of $W(C_{\psi,\phi})$.
\end{proposition}
\begin{proof}
    Since $\phi$ is non-constant, $C_{\psi,\phi}$ is injective. If $\psi$ has a zero at $s_0\in \mathbb{C}_{1/2}$, then the range of $C_{\psi,\phi}$ is contained in the closed subspace of functions that vanish at $s_0$ and hence is not dense. We know that, if a bounded linear operator $T$ is injective and does not have dense range, then $0$ is in the interior of $W(T)$ and thus we get the conclusion. Now suppose $\psi$ has no zeros in $\mathbb{C}_{1/2}$. Since $c_0=0$, we have $\phi\in \mathcal{D}$. By the result from the theory of analytic, almost-periodic functions \cite[pg 131]{favard55455leccons}, a Dirichlet series that converges absolutely in a half-plane $\Re s>\theta$ is not injective on any vertical strip $\alpha<\Re s<\beta$ in this half-plane. Since $\phi$ is non-constant, we can find $a,b\in \mathbb{C}_{1/2}$ such that $a\neq b$ and $\phi(a)=\phi(b)$. Since $\psi$ has no zeros in $\mathbb{C}_{1/2}$, it follows that $\psi(a),\psi(b)\neq 0$ and hence
    $$C_{\psi,\phi}^*(\overline{\psi(b)}K_a-\overline{\psi(a)}K_b)=\overline{\psi(b)}\overline{\psi(a)}K_{\phi(a)}-\overline{\psi(a)}\overline{\psi(b)}K_{\phi(b)}=0.$$
    Thus, $0$ is an eigenvalue of $C_{\psi,\phi}^*$ but not a normal eigenvalue of $C_{\psi,\phi}^*$ since $C_{\psi,\phi}$ is injective. Thus by \cite[Corollary 3.4]{bourdon2002zero}, $0$ does not lie on the boundary of $W(C_{\psi,\phi}^*)$. Therefore, $0$ lies in the interior of $W(C_{\psi,\phi}^*)$ and hence $0$ lies in the interior of $W(C_{\psi,\phi})$.
\end{proof}


\noindent Now, we consider the case when $\phi$ is not of the form $\phi(s)=s+c_1$ where $\Re c_1\geq 0$. This reduces to the case of the composition operators \cite[Theorem 9]{finet2004numerical} when $\psi\equiv 1$.

\begin{theorem}\label{interior main}
     Suppose that $\phi$ is a non-constant $c_0$-symbol and $\psi=\displaystyle\sum_{n=1}^\infty \psi_n n^{-s}\in \mathcal{H}^2$ such that $C_{\psi,\phi}$ is bounded. If $\phi$ is not of the form $\phi(s)=s+c_1$ where $\Re c_1\geq 0$, then $0$ lies in $W(C_{\psi,\phi})$. In particular, if $\psi_1\neq0$, then $0$ lies in the interior of $W(C_{\psi,\phi})$.
\end{theorem}

\begin{proof}
    Let $\phi(s)=c_0s+\eta(s)$, where $c_0\in \mathbb{N}$ and $\eta(s)=\displaystyle\sum_{n=1}^\infty c_n n^{-s} \in\mathcal{D}$ and let $\psi(s)=\displaystyle\sum_{n=1}^\infty \psi_n n^{-s}\in \mathcal{H}^2$. If $c_0=0$, we get the required result from Proposition \ref{c0=0}. Now, suppose $c_0\geq 2$. From \cite{gordon1999composition}, we have for any $n\geq 2$,
    $$n^{-\phi(s)}=n^{-c_0 s}n^{-c_1}\left(1+\sum_{m=2}^\infty a_{m}^{(n)} m^{-s}\right).$$
    Consider $f(s)=d_1 2^{-s}+ d_2 2^{-c_0 s}$ with $|d_1|^2+|d_2|^2=1$. Then
    \begin{align*}
        &C_{\psi,\phi} f(s)\\
        &= \sum_{n=1}^\infty \psi_n n^{-s}\left(d_1 2^{-c_0s}2^{-c_1}\left(1+\sum_{m=2}^\infty a_{m}^{(2)} m^{-s}\right)+d_2 (2^{c_0})^{-c_0s}(2^{c_0})^{-c_1}\left(1+\sum_{m=2}^\infty a_{m}^{(2^{c_0})} m^{-s}\right)\right)
    \end{align*} 
    and hence $\langle C_{\psi,\phi} f,f\rangle=\psi_1 d_1 \overline{d_2}2^{-c_1}$. If $A=\begin{bmatrix}
        0 & 0\\
        2^{-c_1}\psi_1 &0
    \end{bmatrix}$ on $\mathbb{C}^2$, then $W(A)=\{2^{-c_1}\psi_1x\overline{y}:|x|^2+|y|^2=1\}$ and hence $W(A)\subseteq W(C_{\psi,\phi})$. Therefore, from \cite{gustafson1997numerical}, we get, $W(C_{\psi,\phi})$ contains the closed disk centered at origin and radius $|\psi_1|2^{-\Re c_1-1}$ and the result follows. Now, let $c_0=1$ and $\phi(s)=s+c_1+c_r r^{-s}+\cdots $, where $c_r\neq 0$. From \cite{gordon1999composition}, for any $n\geq 2$, we have 
    \begin{align*}
        n^{-\phi(s)}&=n^{-s}n^{-c_1}\prod_{m=r}^\infty\exp(-c_m m^{-s} \log n)\\
        &=n^{-s}n^{-c_1}\prod_{m=r}^\infty\left(1-c_mm^{-s}\log n+ \frac{(-c_mm^{-s}\log n)^2}{2}+\cdots\right)\\
        &=n^{-s}n^{-c_1}(1-c_r r^{-s} \log n+\cdots).
    \end{align*}
    Consider, for any integer $p\geq1$, $g(s)=d_3(r^p)^{-s}+d_4(r^{p+1})^{-s}$ where $|d_3|^2+|d_4|^2=1$. Then
    \begin{align*}
        &C_{\psi,\phi} g(s)\\
        &= \sum_{n=1}^\infty \psi_n n^{-s}\left(d_3 r^{-ps}r^{-pc_1}(1-c_r r^{-s} \log r^p+\cdots) +d_4r^{-(p+1)s}r^{-(p+1)c_1}(1-c_r r^{-s} \log r^{(p+1)}+\cdots) \right)
    \end{align*}
    and hence \begin{align*}
        \langle C_{\psi,\phi} g,g\rangle&=\psi_1|d_3|^2r^{-pc_1}+\psi_1|d_4|^2r^{-(p+1)c_1}-\psi_1d_3\overline{d_4}r^{-pc_1}c_r \log r^p\\
        &=\psi_1 r^{-pc_1}\left(|d_3|^2+|d_4|^2 r^{-c_1}-pc_rd_3\overline{d_4}\log r\right).
    \end{align*}
     If $B_p=\begin{bmatrix}
        1 & 0\\
        -p c_r\log r &r^{-c_1}
    \end{bmatrix}$ on $\mathbb{C}^2$, then $W(B_p)=\{|x|^2+r^{-c_1}|y|^2-p c_r x\overline{y}\log r :|x|^2+|y|^2=1\}$ and hence $\psi_1 r^{-pc_1}W(B_p)\subseteq W(C_{\psi,\phi})$. We take $p$ to be large enough such that $1+|r^{-c_1}|<p |c_r|\log r$, then from \cite[Lemma 10]{finet2004numerical}, it follows that $0$ lies in the interior of $W(B_p)$. Hence, we get the result.

\end{proof}

\begin{example}
    Suppose $\phi(s)=2s+1$ be analytic self-map on $\mathbb{C}_{1/2}$ and $\psi(s)=3^{-s}$ be analytic map on $\mathbb{C}_{1/2}$. Then, $\phi$ is a $c_0-$symbol, $\psi\in \mathcal{M}$ and hence $C_{\psi,\phi}$ is bounded. From Theorem \ref{interior main}, we obtain that $0$ lies in $W(C_{\psi,\phi})$. 
\end{example}

Now, we deal with the case where $\phi(s)=s+c_1$ where $\Re c_1\geq 0$.
\begin{proposition}
    Suppose that $\phi$ is of the form $\phi(s)=s+c_1$ where $\Re c_1\geq 0$ and $\psi(s)=\displaystyle\sum_{n=1}^\infty \psi_n n^{-s}\in \mathcal{H}^2$ such that $C_{\psi,\phi}\in\mathbb{B}(\mathcal{H}^2)$. If $\psi_1=0$, then $0$ lies in $W(C_{\psi,\phi})$.
\end{proposition}

\begin{proof}
    Consider $f(s)=2^{-s}\in \mathcal{H}^2$. Then $\|f\|=1$ and $$\langle C_{\psi,\phi}f,f\rangle=\left\langle\left(\sum_{n=1}^\infty \psi_n n^{-s}\right)2^{-s} 2^{-c_1}, 2^{-s}\right\rangle=\psi_1 2^{-c_1}.$$
    Hence the result follows.
\end{proof}

We now consider the case when $\phi(s)=s+c_1$ where $\Re c_1\geq 0$ and $\psi_1\neq 0$. In that case, $0$ may or may not belong to $W(C_{\psi,\phi})$. This is shown by the following proposition for non-zero constant $\psi$. The result also gives the numerical range of normal weighted composition operators on $\mathcal{H}^2$. 

\begin{proposition}\label{normal}
    Suppose that $\phi$ is of the form $\phi(s)=s+c_1$ where $\Re c_1\geq 0$ and $\psi\equiv a\in \mathbb{C}_{1/2}$ such that $C_{\psi,\phi}$ is bounded. Then
    \begin{enumerate}[(a)]
        \item If $c_1=0$, then $W(C_{\psi,\phi})=\{a\}$.
        \item If $c_1>0$, then $W(C_{\psi,\phi})=(0,a]$.
        \item If $\Re c_1>0$ and $c_1\notin \mathbb{R}$, then $W(C_{\psi,\phi})$ is a closed polygon containing $0$ in its interior.
        \item If $\Re c_1=0$, then $W(C_{\psi,\phi})=\{z:|z|<|a|\}\cup\{a n^{-c_1}:n\geq 1\}$.
    \end{enumerate}
\end{proposition}

\begin{proof}
     From \cite[Proposition 2.5]{yao2021complex}, we have, $C_{\psi,\phi}$ is normal for given $\phi$ and $\psi$. The result follows from \cite[Proposition 6]{finet2004numerical}.
\end{proof}

\noindent The following corollary provides the condition under which the numerical range of a weighted composition operator is closed.
\begin{corollary}
    Suppose that $\phi$ is of the form $\phi(s)=s+c_1$ where $\Re c_1\geq 0,\ c_1\notin\mathbb{R}$ and $\psi$ is a non-zero constant such that $C_{\psi,\phi}$ is bounded. Then $C_{\psi,\phi}$ is compact, and hence $W(C_{\psi,\phi})$ is closed.
\end{corollary}
\begin{proof}
    From \cite[Theorem 2.12, Corollary 2.15]{yao2021complex}, it follows that $C_{\psi,\phi}$ is compact. Proposition \ref{normal} and \cite[Theorem 1]{de1972numerical} implies that $W(C_{\psi,\phi})$ is closed.
\end{proof}

If $\psi$ is non-constant with $\psi_1\neq0$, then $0$ may or may not lie in $W(C_{\psi,\phi})$. This is illustrated from the following examples.

\begin{example}
    Let $\phi(s)=s+2$ be analytic self-map on $\mathbb{C}_{1/2}$ and $\displaystyle\psi(s)=1+\frac{1}{4}2^{-s}$ be analytic map on $\mathbb{C}_{1/2}$. Then $0\notin W(C_{\psi,\phi})$.
\end{example}
\begin{proof}
    Let $f(s)=\displaystyle\sum_{n=1}^\infty f_n n^{-s}\in\mathcal{H}^2$ with $\|f\|=1$. Then 
    \begin{align*}
        \langle C_{\psi,\phi} f,f\rangle=\sum_{n=1}^\infty |f_n|^2 n^{-2}+\frac{1}{4}\sum_{n=1}^\infty f_n \overline{f_{2n}} n^{-2}=A+\frac{1}{4}B
    \end{align*}
    where $A=\displaystyle\sum_{n=1}^\infty |f_n|^2 n^{-2}$ and $B=\displaystyle\sum_{n=1}^\infty f_n \overline{f_{2n}} n^{-2}$. Since $|f_n\overline{f_{2n}}|\leq \displaystyle\frac{|f_n|^2+|f_{2n}|^2}{2}$, therefore we have $$|B|\leq\frac{5A}{2}\ \text{as}\ A>0.$$
    Thus $\displaystyle A+\frac{1}{4}B\neq 0$, hence $0\notin W(C_{\psi,\phi})$. But $\psi\in H^\infty(\mathbb{C}_0)$, thus Corollary \ref{closure} implies that $0\in\overline{W(C_{\psi,\phi})}$ and hence $0$ lies on the boundary of $W(C_{\psi,\phi})$.
\end{proof}

\begin{example}
    Let $\phi(s)=s+2$ be analytic self-map on $\mathbb{C}_{1/2}$ and $\displaystyle\psi(s)=1+4.2^{-s}$ be analytic map on $\mathbb{C}_{1/2}$. Then $0$ lies in the interior of $W(C_{\psi,\phi})$.
\end{example}
\begin{proof}
    Let $f_\theta(s)=\displaystyle\frac{1}{\sqrt{2}}(1+e^{i\theta}2^{-s})$, then $\|f_\theta\|=1$ and $\langle C_{\psi,\phi}f_\theta,f_\theta\rangle=\displaystyle\frac{5}{8}+2e^{-i\theta}$. Since $\theta$ is arbitrary, we get the conclusion.
\end{proof}

In the next result, we summarize the preceding results to characterize the inclusion of $0$ in the numerical range of a weighted composition operator.

\begin{Result}\label{summary1}
    Suppose that $\phi$ is a non-constant $c_0$-symbol with $c_0\neq0$ and $\psi=\displaystyle\sum_{n=1}^\infty \psi_n n^{-s}\in \mathcal{H}^2$ such that $C_{\psi,\phi}$ is bounded, then
    \begin{enumerate}
        \item If $\psi_1=0$, then $0\in W(C_{\psi,\phi})$.
        \item If $\psi_1\neq0$ then the following cases arise:
        \begin{enumerate}
            \item If $\phi$ is not of the form $\phi(s)=s+c_1$ where $\Re c_1\geq 0$, then $0$ lies in the interior of $W(C_{\psi,\phi})$.
            \item If $\phi(s)=s+c_1$ where $\Re c_1\geq 0$, then $0$ may or may not belong to $W(C_{\psi,\phi})$.
        \end{enumerate}
    \end{enumerate}
\end{Result}

\begin{question}\label{ques}
    Let $\phi$ and $\psi$ be defined as in Result \ref{summary1} with $\psi_1=0$. Does $0$ lie in the interior of $W(C_{\psi,\phi})$?
\end{question}
We give partial answer to this question in the following section.

\section{Inclusion of circular or elliptic disk in the numerical range}\label{Sec3}

The results presented in this section are motivated by the work of Gunatillake et al. \cite{gunatillake2014numerical} on weighted composition operators on Hardy space. These determine when the numerical range of a weighted composition operator contains disks centered at the origin. In such cases, $0$ lies in the interior of the numerical range. Also, they provide lower estimate for the numerical radius of $C_{\psi,\phi}$.

\begin{proposition}\label{circular1}
    Suppose that $\phi$ is of the form $\phi(s)=s+c$ where $\Re c\geq 0$ and $\psi(s)=\displaystyle\sum_{n=1}^\infty \psi_n n^{-s}\in \mathcal{H}^2$ with $\psi_1=0$ such that $C_{\psi,\phi}$ is bounded. Then $W(C_{\psi,\phi})$ contains the closed circular disk whose center is at the origin and radius is $|\psi_{j/2} 2^{-c}|/2$ for each even $j>2$.
\end{proposition}
\begin{proof}
    Let $j>2$ be an even integer. We define $M_j$ to be the subspace of $\mathcal{H}^2$ spanned by $e_2(s)=2^{-s}$ and $e_j(s)=j^{-s}$ for each $j$. Therefore,
    $$C_{\psi,\phi}(e_2)=(\psi_2 2^{-s}+\psi_3 3^{-s} +\cdots)2^{-s} 2^{-c}$$
    and
    $$C_{\psi,\phi}(e_j)=(\psi_2 2^{-s}+\psi_3 3^{-s} +\cdots)j^{-s} j^{-c}.$$
    Then the matrix representation of the compression of $C_{\psi,\phi}$ to $M_j$ is given by $$T_j=\begin{bmatrix}
        0 & 0 \\
        \psi_{j/2}2^{-c} &0
    \end{bmatrix}.$$
    From \cite{gustafson1997numerical}, we have $W(T_j)$ is the closed disk centered at the origin and radius $|\psi_{j/2} 2^{-c}|/2$. Since $W(T_j)\subseteq W(C_{\psi,\phi})$ for each $j$, hence the result follows.
\end{proof}

\begin{example}
     Suppose $\phi(s)=s+2$ be analytic self-map on $\mathbb{C}_{1/2}$ and $\psi(s)=3^{-s}+(i+3)4^{-s}$ be analytic map on $\mathbb{C}_{1/2}$. Then $\phi$ is a $c_0-$symbol, $\psi\in \mathcal{M}$ and hence $C_{\psi,\phi}$ is bounded. From Proposition \ref{circular1}, $W(C_{\psi,\phi})$ contains a circular disk with center at origin and radius $\frac{\sqrt{10}}{32}$. Also, we get that $w(C_{\psi,\phi})\geq \frac{\sqrt{10}}{32}$.
\end{example}

\begin{proposition}\label{circular2}
    Suppose that $\phi(s)=c_0s+\eta(s)$, where $c_0 (\geq 2)\in \mathbb{N}$ and $\eta(s)=\displaystyle\sum_{n=1}^\infty c_n n^{-s} \in\mathcal{D}$ and let $\psi(s)=\displaystyle\sum_{n=m}^\infty \psi_n n^{-s}\in \mathcal{H}^2$ where $m\geq2$ such that $C_{\psi,\phi}\in \mathbb{B}(\mathcal{H}^2)$. Then $W(C_{\psi,\phi})$ contains a circular disk of radius $|\psi_m|/2$ centered at the origin.
\end{proposition}
\begin{proof}
    Let $f_m(s)=\displaystyle\frac{\lambda+m^{-s}}{\sqrt{2}}$ with $|\lambda|=1$. Then $\|f_m\|=1$ for each $m\geq 2$ and from \cite{gordon1999composition}, we get,
    $$(f_m\circ \phi)(s)=\frac{1}{\sqrt{2}}\left(\lambda+m^{-c_0s}m^{c_1}\left(1+\sum_{l=2}^\infty d_l^{(m)}l^{-s}\right)\right).$$ Thus
    \begin{eqnarray*}
        &&\langle C_{\psi,\phi} f_m,f_m\rangle\\
        &=&\langle\psi\cdot(f_m\circ \phi),f_m\rangle\\
        &=&\frac{1}{2}\left\langle \left(\psi_m m^{-s}+\psi_{m+1} (m+1)^{-s}+\ldots\right)\left(\lambda+m^{-c_0s}m^{c_1}\left(1+\sum_{l=2}^\infty d_l^{(m)}l^{-s}\right)\right),\lambda+m^{-s}\right\rangle\\
        &=&\frac{1}{2}\Big\langle \lambda\psi_m m^{-s}+(\text{terms containing}\ (m+k)^{-s}\  \text{where}\ k\in\mathbb{N})+\\
        &&(\text{terms containing}\ m^{-k}\  \text{for}\ k>s),\lambda+m^{-s}\Big\rangle\\
        &=&\frac{1}{2}\lambda \psi_m.
    \end{eqnarray*}
    Since $|\lambda|=1$ and $\|f_m\|=1$, we get the result.
\end{proof}

\begin{example}
    Suppose $\phi(s)=4s+2^{-s}$ be analytic self-map on $\mathbb{C}_{1/2}$ and $\psi(s)=(i+2)3^{-s}$ be analytic map on $\mathbb{C}_{1/2}$. Then $\phi$ is a $c_0-$symbol, $\psi\in \mathcal{M}$ and hence $C_{\psi,\phi}$ is bounded. Thus from Proposition \ref{circular2}, we conclude that a circular disk of radius $\frac{\sqrt{5}}{2}$ centered at origin is contained in $W(C_{\psi,\phi})$. Also $w(C_{\psi,\phi})\geq \frac{\sqrt{5}}{2}$.
\end{example}

\noindent The following results determine when certain elliptic disks are contained in the numerical range of a weighted composition operator.

\begin{proposition}\label{ellipse}
    Let $\phi$ be of the form $\phi(s)=s+c$ where $\Re c\geq 0$ and $\psi(s)=\displaystyle\sum_{n=1}^\infty \psi_n n^{-s}\in \mathcal{H}^2$ such that $C_{\psi,\phi}\in \mathbb{B}(\mathcal{H}^2)$. Suppose $\psi_j\neq 0$ for some $j\ (\geq 2)\in \mathbb{N}$, then $W(C_{\psi,\phi})$ contains the ellipse with foci at $\psi_1$ and $\psi_1 j^{-c}$, and with minor axis $|\psi_j|$ and major axis $\sqrt{|\psi_1-\psi_1j^{-c}|^2+|\psi_j|^2}$. In particular, if $\psi_1=0$, then $W(C_{\psi,\phi})$ contains the disk center at $0$ and radius $|\psi_j|$.
\end{proposition}

\begin{proof}
    Let $Q_j=$span$\{e_1,e_j\}$ where $e_1(s)=1$ and $e_j(s)=j^{-s}$ for $j\geq 2$. Now,
    $$C_{\psi,\phi}(e_1)=\psi_1+\psi_2 2^{-s}+\psi_3 3^{-s} +\cdots$$
    and
    $$C_{\psi,\phi}(e_j)=(\psi_1+\psi_2 2^{-s}+\psi_3 3^{-s} +\cdots)j^{-s} j^{-c}.$$
    Let $A_j$ be the matrix representation of the compression of $C_{\psi,\phi}$ to $Q_j$. Then
    $$A_j=\begin{bmatrix}
        \psi_1 &0\\
        \psi_j &\psi_1 j^{-c}
    \end{bmatrix}.$$
    From \cite[Theorem 1.5]{gau2021numerical}, we get that $W(A_j)$ is the ellipse with foci at $\psi_1$ and $\psi_1 j^{-c}$, and with minor axis $|\psi_j|$ and major axis $\sqrt{|\psi_1-\psi_1j^{-c}|^2+|\psi_j|^2}$. Since $W(A_j)\subseteq W(C_{\psi,\phi})$, hence the result follows.
\end{proof}

\begin{example}
    Suppose $\phi(s)=s+2$ be analytic self-map on $\mathbb{C}_{1/2}$ and $\psi(s)=1+(i+2)2^{-s}+(2i+1)3^{-s}$ be analytic map on $\mathbb{C}_{1/2}$. Then $\phi$ is a $c_0-$symbol, $\psi\in \mathcal{M}$ and hence $C_{\psi,\phi}$ is bounded. From Proposition \ref{ellipse}, $W(C_{\psi,\phi})$ contains the convex hull of the elliptic disks with foci at $(1,0),(\frac{1}{4},0)$, minor axis $\sqrt{5}$, major axis $\displaystyle \frac{\sqrt{89}}{4}$ and foci at $(1,0),(\frac{1}{9},0)$, minor axis $\sqrt{5}$, major axis $\displaystyle \frac{\sqrt{469}}{9}$.
\end{example}

\begin{proposition}
    Let $\displaystyle\phi(s)=s+\frac{2\pi i}{\log n}$ where $n\in \mathbb{N}\setminus \{1\}$ and $\psi(s)=\displaystyle\sum_{m=1}^\infty \psi_m m^{-s}\in \mathcal{H}^2$ such that $C_{\psi,\phi}\in \mathbb{B}(\mathcal{H}^2)$. Let $a\geq 2$ be a positive integer such that $\displaystyle \psi_n\psi_{n^a}\psi_{n^{a-1}}=0$ but at least one of the three terms is non-zero. Then circular disk of radius $\frac{1}{2}\sqrt{|\psi_n|^2+|\psi_{n^a}|^2+|\psi_{n^{a-1}}|^2}$ centered at $\psi_1$ is contained in  $W(C_{\psi,\phi})$.
\end{proposition}

\begin{proof}
    Let $M_{n,a}$ be the subspace of $\mathcal{H}^2$ spanned by $$e_1(s)=1,\ e_n(s)=n^{-s}\ \text{and}\ e_{n^a}(s)=(n^a)^{-s}.$$
    Let $T_{n,a}$ be the compression of $C_{\psi,\phi}$ to $M_{n,a}$. Now,
    $$C_{\psi,\phi}(e_1)=\psi_1+\psi_2 2^{-s}+\psi_3 3^{-s} +\cdots,$$
    \begin{align*}
        C_{\psi,\phi}(e_n)&=(\psi_1+\psi_2 2^{-s}+\psi_3 3^{-s} +\cdots)n^{-s} n^{-\frac{2\pi i}{\log n}}\\
        &=(\psi_1+\psi_2 2^{-s}+\psi_3 3^{-s} +\cdots)n^{-s}
    \end{align*}
     and 
     \begin{align*}
         C_{\psi,\phi}(e_{n^a})&=(\psi_1+\psi_2 2^{-s}+\psi_3 3^{-s} +\cdots)(n^a)^{-s}n^{-\frac{2\pi ia}{\log n}}\\
         &=(\psi_1+\psi_2 2^{-s}+\psi_3 3^{-s} +\cdots)(n^a)^{-s}.
     \end{align*}
     Thus the matrix representation of $T_{n,a}$ is given by
     $$\begin{bmatrix}
         \psi_1 & 0 & 0\\
         \psi_n &\psi _1 & 0\\
         \psi_{n^a}& \psi_{n^{a-1}} & \psi_1
     \end{bmatrix}.$$
     From \cite[Theorem 4.1]{keeler1997numerical}, $W(T_{n,a})$ is the circular disk centered at $\psi_1$ and radius $$\frac{1}{2}\sqrt{|\psi_n|^2+|\psi_{n^a}|^2+|\psi_{n^{a-1}}|^2}.$$
     Since $W(T_{n,a})\subseteq W(C_{\psi,\phi})$, we obtain the result.
\end{proof}

The results obtained in this section provide the following partial answers to Question \ref{ques} raised in Section \ref{Sec2}.

\begin{Result}\label{summary2}
    Suppose that $\phi$ is a non-constant $c_0$-symbol with $c_0\neq0$ and $\psi=\displaystyle\sum_{n=1}^\infty \psi_n n^{-s}\in \mathcal{H}^2$ with $\psi_1=0$ such that $C_{\psi,\phi}$ is bounded. Then $0$ belongs to the interior of $W(C_{\psi,\phi})$ for the following cases:
    \begin{enumerate}
        \item $\phi(s)=s+c_1$ where $\Re c_1\geq 0$.
        \item $\phi(s)=c_0s+\eta(s)$, where $c_0 (\geq 2)\in \mathbb{N}$ and $\eta(s)\in\mathcal{D}$.
    \end{enumerate}
\end{Result}

Cases (1) and (2) follow from Proposition \ref{ellipse} and \ref{circular2}, respectively. We end this section with the following question.
\begin{question}
    Let $\phi$ and $\psi$ be defined as in Result \ref{summary2}. If $\phi(s)=s+c_1+c_r r^{-s}+\cdots $, where $c_r\neq 0$, does $0$ lie in the interior of $W(C_{\psi,\phi})$?
\end{question}

\section{Invariant subspaces of weighted composition operators}\label{Sec4}
A closed subspace $M$ of a complex separable Hilbert space $H$ is called an invariant subspace of $T\in\mathbb{B}(H)$, if $Tx\in M$ for every $x\in M$, i.e., $TM\subseteq M$. Also, $M$ is said to be reducing subspace of $T$ if $M$ is invariant under both $T$ and $T^*$, or equivalently, both $M$ and $M^\perp$ are invariant under $T$. For $T\in\mathbb{B}(H)$, the lattice of all invariant subspaces of $T$ is denoted by Lat $T$. An operator $T\in\mathbb{B}(H)$ is said to be reductive if every invariant subspace of $T$ is reducing. In this section, we give some results concerning lattices of all invariant subspaces of weighted composition operators defined on Hardy space of Dirichlet series. Wang and Yao \cite{wang2015invariant} derived these results for composition operators. In order to view them for weighted composition operators, we need the following lemmas.

\begin{lemma}\label{lim kw}
    For $\omega\in \mathbb{C}_\delta=\{s\in \mathbb{C}: \Re s>\delta\}$ for sufficiently large $\delta$, we have $\displaystyle\lim_{\Re\omega\to+\infty} K_\omega=1.$
\end{lemma}
\begin{proof}
    From the definition, it follows that $K_\omega(s)=\displaystyle\sum_{n=1}^\infty n^{-(\overline{\omega}+s)}$ for $s\in\mathbb{C}_{1/2}$. So, \begin{eqnarray*}
    K_\omega(s)-1=\sum_{n=2}^\infty n^{-(\overline{\omega}+s)}\ \text{and}\ \displaystyle\|K_\omega -1\|^2=\sum_{n=2}^\infty |n^{-\overline{ \omega}}|^2=\sum_{n=2}^\infty n^{-2\Re \omega}.
    \end{eqnarray*}
    Now, for $n\geq 2$, we have $n^{-2\Re{\omega}}\to 0$ as $\Re\omega\to +\infty$. Therefore, $\displaystyle\sum_{n=2}^\infty n^{-2\Re{\omega}}\to 0$ as $\Re\omega\to +\infty$ follows from uniform convergence of the series. Thus, $\displaystyle\lim _{\Re\omega\to+\infty}\|K_\omega-1\|= 0$ which implies $\displaystyle\lim_{\Re\omega\to+\infty} K_\omega=1.$
\end{proof}

\begin{lemma}\label{adj1}
    Let $\phi(s)=c_0s+\eta(s)$, where $c_0\in \mathbb{N}$ and $\eta(s)=\displaystyle\sum_{k=1}^\infty c_k k^{-s} \in\mathcal{D}$ and let $\psi(s)=\displaystyle\sum_{n=1}^\infty \psi_n n^{-s}\in \mathcal{H}^2$ such that $C_{\psi,\phi}\in \mathbb{B}(\mathcal{H}^2)$. Then the following results hold:
    \begin{enumerate}[(i)]
        \item If $c_0=0$, then $C_{\psi,\phi}^*1=\overline{\psi_1}\displaystyle\sum_{m=1}^\infty m^{-\overline{c_1}}m^{-s}$,
        \item If $c_0\geq1$, then $C_{\psi,\phi}^*1=\overline{\psi_1}$.
    \end{enumerate}
\end{lemma}

\begin{proof}
    \begin{enumerate}[(i)]
        \item As $c_0=0$, by \cite{gordon1999composition}, we have
        \begin{align*}
            C_{\psi,\phi}f(s)&=\psi(s)(f(\phi(s))\\
            &=\displaystyle\sum_{n=1}^\infty \psi_n n^{-s}\left(\sum_{m=1}^\infty a_m m^{-c_1} m^{-\sum_{k=2}^\infty c_k k^{-s}}\right)\ \text{where}\ f(s)=\displaystyle\sum_{m=1}^\infty a_m m^{-s}\\
            &= \displaystyle\sum_{n=1}^\infty \psi_n n^{-s}\sum_{m=1}^\infty a_m m^{-c_1}\prod_{k=2}^\infty\left(1+\sum_{l=1}^\infty \frac{(-c_k\log m)^l}{l!}k^{-ls}\right)
        \end{align*}
        which holds in half-plane of absolute convergence of the series $\sum_{k=1}^\infty c_k k^{-s}$.
        Hence 
        \begin{align*}
            \langle C_{\psi,\phi}^*1,f\rangle =\langle1,C_{\psi,\phi}f\rangle
            &=\lim_{\Re\omega\to+\infty}\langle K_\omega,C_{\psi,\phi}f\rangle \ (\text{from Lemma}\ \ref{lim kw})\\
            &=\lim_{\Re\omega\to+\infty}\langle C_{\psi,\phi}^*K_\omega,f\rangle\\
            &=\lim_{\Re\omega\to+\infty}\langle \overline{\psi(\omega)}K_{\phi(\omega)},f\rangle\\
            &=\lim_{\Re\omega\to+\infty}\overline{\psi(\omega)}\overline{\langle f,K_{\phi(\omega)}\rangle}\\
            &=\lim_{\Re\omega\to+\infty}\overline{\psi(\omega)f(\phi(\omega))}\\
            &=\overline{\psi_1}\sum_{m=1}^\infty\overline{a_m}m^{-\overline{c_1}}\\
            &=\langle\overline{\psi_1}\displaystyle\sum_{m=1}^\infty m^{-\overline{c_1}}m^{-s},f\rangle.
        \end{align*}
        Since $f\in \mathcal{H}^2$ is arbitrary, we have $C_{\psi,\phi}^*1=\overline{\psi_1}\displaystyle\sum_{m=1}^\infty m^{-\overline{c_1}}m^{-s}$.
        \item For $c_0\geq 1$, we know that $+\infty$ is a fixed point of $\phi$ and hence
        $$C_{\psi,\phi}^*1=\lim_{\Re\omega\to+\infty}C_{\psi,\phi}^*K_\omega=\lim_{\Re\omega\to+\infty}\overline{\psi(\omega)}K_{\phi(\omega)}=\overline{\psi_1}.$$
    \end{enumerate}
\end{proof}

\noindent The following theorem gives some sufficient conditions for the case when the lattice of invariant subspaces of one weighted composition operator in contained in that of another operator.

\begin{theorem}
    Let $\displaystyle\phi_1(s)=c_0^{(1)}s+\sum_{k=1}^\infty c_k^{(1)} k^{-s}$ and $\displaystyle\phi_2(s)=c_0^{(2)}s+\sum_{k=1}^\infty c_k^{(2)} k^{-s}$ be $c_0$-symbols and let $\psi_1(s)=\displaystyle\sum_{n=1}^\infty \psi_n^{(1)} n^{-s}$ and $\psi_2(s)=\displaystyle\sum_{n=1}^\infty \psi_n^{(2)} n^{-s}$ be in $\mathcal{H}^2$ such that $C_{\psi_1,\phi_1},C_{\psi_2,\phi_2}\in \mathbb{B}(\mathcal{H}^2)$. If Lat $C_{\psi_1,\phi_1}\subseteq$ Lat $C_{\psi_2,\phi_2}$, then we have the following.
    \begin{enumerate}[(i)]
        \item If $c_0^{(1)}\geq 1$ and $\psi_1^{(2)}\neq 0$, then $c_0^{(2)}\geq 1$.
        \item If $c_0^{(1)}=0$ and $\omega\in \mathbb{C}_{1/2}$ is a fixed point of $\phi_1$, then $\omega$ is also a fixed point of $\phi_2$, provided $\psi_2(\omega)\neq 0$.
    \end{enumerate}
\end{theorem}

\begin{proof}
\begin{enumerate}[(i)]
    \item  If $c_0^{(1)}\geq 1$, then from Lemma \ref{adj1}, we have $C_{\psi_1,\phi_1}^*1=\overline{\psi_1^{(1)}}$ and hence, span$\{1\}\in$ Lat $C_{\psi_1,\phi_1}^*$. Therefore, span$\{1\}^\perp\in$ Lat $C_{\psi_1,\phi_1}$. Since Lat $C_{\psi_1,\phi_1}\subseteq$ Lat $C_{\psi_2,\phi_2}$, we have span$\{1\}^\perp\in$ Lat $C_{\psi_2,\phi_2}$ and thus, span$\{1\}\in$ Lat $C_{\psi_2,\phi_2}^*$. If possible, let $c_0^{(2)}=0$. Then from Lemma \ref{adj1}, $C_{\psi_2,\phi_2}^*1=\overline{\psi_1^{(2)}}\displaystyle\sum_{n=1}^\infty n^{-\overline{c_1^{(2)}}}n^{-s}\notin$ span$\{1\}$ which leads to a contradiction. Thus $c_0^{(2)}\geq 1$.
    \item Suppose $c_0^{(1)}=0$ and $\phi_1(\omega)=\omega$ where $\omega\in \mathbb{C}_{1/2}$. Therefore,
    $$C_{\psi_1,\phi_1}^*K_\omega=\overline{\psi_1(\omega)}K_{\phi_1(\omega)}=\overline{\psi_1(\omega)}K_{\omega}\in \ \text{span}\{K_\omega\}$$
    and hence span$\{K_\omega\}\in$ Lat $C_{\psi_1,\phi_1}^*\subseteq$ Lat $C_{\psi_2,\phi_2}^*$. If $\phi_2(\omega)\neq \omega$, then $K_\omega$ and $K_{\phi_2(\omega)}$ are linearly independent. But $$\overline{\psi_2(\omega)}K_{\phi_2(\omega)}=C_{\psi_2,\phi_2}^*K_\omega\in \text{span}\{K_\omega\}.$$
    Thus $\phi_2(\omega)=\omega.$
\end{enumerate}
   
\end{proof}

\noindent The next theorem considers the images of two weighted composition operators.

\begin{theorem}
     Let $\phi_1(s)=c_0^{(1)}s+\sum_{k=1}^\infty c_k^{(1)} k^{-s}$ and $\phi_2(s)=c_0^{(2)}s+\sum_{k=1}^\infty c_k^{(2)} k^{-s}$ be $c_0$-symbols and let $\psi_1(s)=\displaystyle\sum_{m=1}^\infty \psi_m^{(1)} m^{-s}$ and $\psi_2(s)=\displaystyle\sum_{m=1}^\infty \psi_m^{(2)} m^{-s}$ be in $\mathcal{H}^2$ such that $C_{\psi_1,\phi_1},C_{\psi_2,\phi_2}\in \mathbb{B}(\mathcal{H}^2)$. If $C_{\psi_1,\phi_1}(\mathcal{H}^2)\subseteq C_{\psi_2,\phi_2}(\mathcal{H}^2)$, then the following statements hold:
     \begin{enumerate}[(i)]
         \item If $c_0^{(1)}=0$ and $\psi_2^{(2)}=0$, then $c_0^{(2)}\leq 1$.
         \item If $c_0^{(1)}\geq 1$, $\psi_1^{(1)}\neq 0$ and $\psi_{2^{c_0^{(1)}}}^{(2)}=0$, then $c_0^{(2)}\leq c_0^{(1)}.$
     \end{enumerate}
\end{theorem}

\begin{proof}
    \begin{enumerate}[(i)]
        \item Let $c_0^{(1)}=0$. If possible, let $c_0^{(2)}> 1$. Then $n^{c_0^{(2)}}\geq n^2\geq 2^2$ for all $n\geq 2$. Now,
        \begin{align*}
            \langle 2^{-s},C_{\psi_2,\phi_2}(n^{-s})\rangle=\left\langle 2^{-s}, \sum_{m=1}^\infty \psi_m^{(2)} m^{-s}\left(n^{-c_0^{(2)}s}n^{-c_1^{(2)}}n^{-\sum_{k=2}^\infty c_k^{(2)} k^{-s}}\right)\right\rangle=0
        \end{align*}
        for all $n\geq2$. For $n=1$,
        $$\langle2^{-s},C_{\psi_2,\phi_2}(1)\rangle=\psi_2^{(2)}=0.$$
        Therefore, $2^{-s}\in (C_{\psi_2,\phi_2}\mathcal{H}^2)^\perp$ which implies $2^{-s}\in (C_{\psi_1,\phi_1}\mathcal{H}^2)^\perp$. But for some $n\in \mathbb{N}$,
        \begin{align*}
             \langle 2^{-s},C_{\psi_1,\phi_1}(n^{-s})\rangle&=\left\langle 2^{-s}, \sum_{m=1}^\infty \psi_m^{(1)} m^{-s}\left(n^{-c_1^{(1)}}n^{-\sum_{k=2}^\infty c_k^{(1)} k^{-s}}\right)\right\rangle\\
             &=\left\langle 2^{-s}, \left(\sum_{m=1}^\infty \psi_m^{(1)} m^{-s}\right) n^{-c_1^{(1)}}\prod_{k=2}^\infty\left(1+\sum_{l=1}^\infty\frac{(-c_k^{(1)}\log n)^l}{l!}k^{-ls} \right)\right\rangle\\
             &=-\overline{\psi_1^{(1)}n^{-c_1^{(1)}}c_2^{(1)}\log n}+\overline{\psi_2^{(1)}n^{-c_1^{(1)}}}\\
             &\neq 0.
        \end{align*}
        which is a contradiction. Hence $c_0^{(2)}\leq 1$.

        \item Let us assume that $c_0^{(2)}> c_0^{(1)}$. Then $n^{c_0^{(2)}}> n^{c_0^{(1)}}\geq 2^{c_0^{(1)}}$ for all $n\geq 2$. Now,
        \begin{align*}
            \langle 2^{-c_0^{(1)}s},C_{\psi_2,\phi_2}(n^{-s})\rangle=\left\langle 2^{-c_0^{(1)}s}, \sum_{m=1}^\infty \psi_m^{(2)} m^{-s}\left(n^{-c_0^{(2)}s}n^{-c_1^{(2)}}n^{-\sum_{k=2}^\infty c_k^{(2)} k^{-s}}\right)\right\rangle=0
        \end{align*}
        for all $n\geq2$. For $n=1$,
        $$\langle2^{-c_0^{(1)}s},C_{\psi_2,\phi_2}(1)\rangle=\psi_{2^{c_0^{(1)}}}^{(2)}=0.$$
        Therefore, $2^{-c_0^{(1)}s}\in (C_{\psi_2,\phi_2}\mathcal{H}^2)^\perp$ which implies $2^{-c_0^{(1)}s}\in (C_{\psi_1,\phi_1}\mathcal{H}^2)^\perp$. But
        \begin{align*}
            \langle 2^{-c_0^{(1)}s},C_{\psi_1,\phi_1}(2^{-s})\rangle&=\left\langle 2^{-c_0^{(1)}s}, \sum_{m=1}^\infty \psi_m^{(1)} m^{-s}\left(2^{-c_0^{(1)}s}2^{-c_1^{(1)}}2^{-\sum_{k=2}^\infty c_k^{(1)} k^{-s}}\right)\right\rangle\\
            &=\overline{2^{-c_1^{(1)}} \psi_1^{(1)}}\\
            &\neq 0
        \end{align*}
        since $\psi_1^{(1)}\neq 0$. This contradicts that $C_{\psi_1,\phi_1}(\mathcal{H}^2)\subseteq C_{\psi_2,\phi_2}(\mathcal{H}^2)$. Hence $c_0^{(2)}\leq c_0^{(1)}$.
        \end{enumerate}
\end{proof}

\noindent Now, we prove our main result which characterizes the reductive weighted composition operators on Hardy space of Dirichlet series.

\begin{theorem}\label{reductive}
    Let $\displaystyle\phi(s)=c_0s+\sum_{k=1}^\infty c_k k^{-s}$ with $c_0(\in\mathbb{N})\geq 1$ and $\psi(s)=\displaystyle\sum_{m=1}^\infty \psi_m m^{-s}\in \mathcal{H}^2$ with $\psi_1\neq 0$ such that $C_{\psi,\phi}\in \mathbb{B}(\mathcal{H}^2)$. Then $C_{\psi,\phi}$ is reductive if and only if $\phi(s)=s+c_1$ with $\Re c_1\geq 0$ and $\psi$ is a non-zero constant function.
\end{theorem}

\begin{proof}
    The sufficiency is due to \cite[Theorem 3.4]{wang2015invariant}. Conversely, let $C_{\psi,\phi}$ be reductive. Now, for $n\geq1$, consider the subspace $\mathcal{M}_n=\overline{\text{span}\ \{k^{-s}:k\geq n\}}$. For any $f\in \mathcal{M}_n$, let $\displaystyle f(s)=\sum_{p=n}^\infty b_p p^{-s}$. Then
    \begin{align*}
        C_{\psi,\phi}f(s)&=\sum_{m=1}^\infty \psi_m m^{-s}\left(\sum_{p=n}^\infty b_p p^{-c_0 s}p^{-c_1} p^{\sum_{k=2}^\infty c_k k^{-s}}\right)\\
        &=\sum_{m=1}^\infty \psi_m m^{-s}\left(\sum_{p=n}^\infty b_p p^{-c_0 s}p^{-c_1}\prod_{k=2}^\infty\left(1+\sum_{l=1}^\infty\frac{(-c_k\log p)^l}{l!}k^{-ls} \right)\right)
    \end{align*}
     from \cite{gordon1999composition}. Since $c_0\neq 0$, therefore $C_{\psi,\phi}f\in\mathcal{M}_n$ and hence $\mathcal{M}_n\in$ Lat $C_{\psi,\phi}$ for all $n\geq 1$. Also $C_{\psi,\phi}$ is reductive, so $\mathcal{M}_n\in$ Lat $C_{\psi,\phi}^*$ for all $n\geq 1$. Hence $C_{\psi,\phi}^* 2^{-c_0 s}\in \mathcal{M}_{2^{c_0 }}$. If possible, let $c_0>1$. Then for any $q$ with $q<2^{c_0}$, we have $q^{-s}\in\mathcal{M}_{2^{c_0 }}^\perp$. But
     \begin{align*}
         \langle C_{\psi,\phi}^* 2^{-c_0 s},2^{-s}\rangle&=\langle 2^{-c_0 s}, \psi(s) 2^{-\phi(s)}\rangle\\
         &=\left\langle 2^{-c_0 s},\left(\sum_{m=1}^\infty \psi_m m^{-s}\right) 2^{-c_0 s} 2^{-c_1}\prod_{k=2}^\infty\left(1+\sum_{l=1}^\infty\frac{(-c_k\log 2)^l}{l!}k^{-ls} \right)\right\rangle\\
         &= \overline{\psi_1}2^{-\overline{c_1}}\\
         &\neq 0
     \end{align*}
     which is a contradiction. Thus $c_0=1$. Now, let $\phi(s)=s+c_1+\sum_{k=2}^\infty c_k k^{-s}$. Since $\mathcal{M}_n\in$ Lat $C_{\psi,\phi}$ for all $n\geq 1$ and $C_{\psi,\phi}$ is reductive, therefore $\mathcal{M}_n^\perp\in$ Lat $C_{\psi,\phi}$ for all $n\geq 2$, where $\mathcal{M}_n^\perp=\text{span}\ \{1,2^{-s},\cdots,(n-1)^{-s}\}$ for $n\geq 2$. For $n=2$, $C_{\psi,\phi}1\in \text{span}\{1\}$ implies that $\psi_m=0$ for all $m\geq 2$. Consider $n\geq3$. Since $\mathcal{M}_n^\perp\in$ Lat $C_{\psi,\phi}$, therefore we have,
     \begin{align*}
         C_{\psi,\phi}((n-1)^{-s})&=\psi(s)(n-1)^{-\phi(s)}\\
         &=\psi_1\left((n-1)^{-s}(n-1)^{-c_1}\prod_{k=2}^\infty\left(1+\sum_{l=1}^\infty\frac{(-c_k\log p)^l}{l!}k^{-ls}\right)\right)\\
         &\in \mathcal{M}_n^\perp
     \end{align*}
     which implies $c_k=0$ for all $k\geq 2$. Therefore, $\phi(s)=s+c_1$ with $\Re c_1\geq 0$ and $\psi(s)=\psi_1$.

\end{proof}

\begin{remark}
    From Theorem \ref{reductive} and \cite[Proposition 2.5]{yao2021complex}, we have, for $\displaystyle\phi(s)=c_0s+\sum_{k=1}^\infty c_k k^{-s}$ with $c_0(\in\mathbb{N})\geq 1$ and $\psi(s)=\displaystyle\sum_{m=1}^\infty \psi_m m^{-s}\in \mathcal{H}^2$ with $\psi_1\neq 0$, $C_{\psi,\phi}$ is normal if and only if it is reductive and in that case, $W(C_{\psi,\phi})$ is given by Proposition \ref{normal}.
\end{remark}

\noindent Now, we show the non-existence of non-trivial finite dimensional invariant subspace of some weighted composition operator on $\mathcal{H}^2$. This result reduces to \cite[Theorem 3.6]{wang2015invariant} when $\psi\equiv 1$. 

\begin{theorem}
    Let $\phi(s)=c_0s+\eta(s)$ be a $c_0$-symbol, where $c_0 (>1)\in \mathbb{N}$ and let $\psi(s)=\displaystyle\sum_{n=1}^\infty \psi_n n^{-s}\in \mathcal{H}^2$ such that $C_{\psi,\phi}\in \mathbb{B}(\mathcal{H}^2)$. Suppose $E(\lambda)$ denotes the eigen space corresponding to the eigenvalue $\lambda$ of $C_{\psi,\phi}$. Then there does not exist any non-trivial finite dimensional invariant subspace $M\nsubseteq E(\psi_1)$ of $C_{\psi,\phi}$ such that $E(\psi_1)\cap M$ is a reducing subspace of $C_{\psi,\phi}$.
\end{theorem}

\begin{proof}
     Let $M\nsubseteq E(\psi_1)$ be any non-trivial finite dimensional invariant subspace of $C_{\psi,\phi}$ such that $E(\psi_1)\cap M$ is a reducing subspace of $C_{\psi,\phi}$. From \cite[Lemma 5.3]{maofa2018weighted}, for $c_0>1$, the point spectrum of $C_{\psi,\phi}$ is a subset of $\{\psi_1\}$. Since $M$ is finite dimensional, therefore, the operator $\restr{C_{\psi,\phi}}{M}$ has eigenvalue $\psi_1$ with eigenvector in $M$. Consider $K=E(\psi_1)\cap M\subset M$ and let $N=M\cap K^\perp$. As $K$ is reducing subspace of $C_{\psi,\phi}$, it follows that $K^\perp$ is invariant under $C_{\psi,\phi}$ which implies $N\in$ Lat$(C_{\psi,\phi})$. Also $N$ is finite dimensional, so $\restr{C_{\psi,\phi}}{N}$ has an eigenvalue other than $\psi_1$ which contradicts \cite[Lemma 5.3]{maofa2018weighted}.
\end{proof}

\section{Additional observations}\label{Sec5}
In this section, we give some results on the Davis-Wielandt shell (DW-shell) of composition operators on Hardy space of Dirichlet series along with some natural questions. Introduced by Davis \cite{davis1968shell} and Wielandt \cite{wielandt1955eigenvalues}, the DW-shell $DW(T)\subset \mathbb{C}\times\mathbb{R}$ of an operator $T\in B(H)$ is defined as
$$ DW(T)= \{(\langle Tx , x \rangle,\langle Tx , Tx \rangle)  : x \in H ,\ \| x \| =1 \}. $$
It follows that the projection of the set $DW(T)$ onto the first co-ordinate is the numerical range $W(T)$. Hence, $DW(T)$ provides additional information about the operator. Moreover, in finite dimensional case, the DW-shell completely characterizes the normal operators. An $n\times n$ complex matrix $A$ is normal if and only if $DW(A)$, regarded as a subset of $\mathbb{C}\times\mathbb{R}\simeq\mathbb{R}^3$, is a polyhedron \cite{li2009davis}. The DW-shell of composition operators was studied by Liu and Liu \cite{liu2026davis} on Hardy space. Motivated by them, we analyze the same for the Hardy space of Dirichlet series. We start with the DW-shell of composition operators with constant symbol.

\begin{proposition}
    Let $\phi\equiv c_1(\in\mathbb{C}_{1/2})$. Then $DW(C_\phi)$ is the region
    $$\{(x_1,x_2,x_3):c^2(x_1^2+x_2^2)+x_3^2-2x_1x_3-(c^2-1)x_3\leq 0\}$$
    where $c=\|K_{c_1}\|=(\zeta(2\Re c_1))^{1/2}$.
\end{proposition}

\begin{proof}
    Consider the functions $\displaystyle f(s)=\frac{K_{c_1}(s)-c^2}{c\sqrt{c^2-1}}$ and $g(s)=1$, where $c=\|K_{c_1}\|>1$. It follows that $C_\phi f=0$ and $C_\phi g=1$, therefore, $f$ and $g$ are unit eigenvectors corresponding to the eigenvalues $0$ and $1$, respectively. By Gram-Schmidt orthogonalization process, we get the unit vectors $\displaystyle u_1(s)=\frac{K_{c_1}(s)-c^2}{c\sqrt{c^2-1}}$ and $\displaystyle u_2(s)=\frac{K_{c_1}(s)}{c}$. It is easy to check that $M=$\ span$\{u_1,u_2\}$ is a reducing subspace of $C_\phi$. Thus $C_\phi$ is unitarily equivalent to $T_1\oplus T_2$, where $T_1$ and $T_2$ are the restrictions of $C_\phi$ to $M$ and $M^\perp$, respectively. Also, the matrix representation of $T_1$ is given by $$T_1=\begin{bmatrix}
         0 & -\sqrt{c^2-1}\\
         0 & 1
     \end{bmatrix}.$$ Since $\text{rank}(C_\phi)=1=\text{rank}(T_1)$, therefore rank$(T_2)=0$ implies $T_2=0$. Thus, $DW(C_\phi)$ is the convex hull of $\{(0,0,0)\}$ and $DW(T_1)$. Hence, it is sufficient to discuss $DW(T_1)$. Any unit vector $u\in \mathbb{C}^2$ can be parameterized as $u=\begin{bmatrix}
        \sqrt{1-r^2}e^{i\theta_1}\\
        re^{i\theta_2}
    \end{bmatrix}$ where $r\in [0,1]$ and $\theta_1,\theta_2\in[0,2\pi)$. Then we get
    $$\langle T_1u,u\rangle=-\sqrt{c^2-1}r\sqrt{1-r^2}e^{i(\theta_2-\theta_1)}+r^2\ \text{and}\ \langle T_1u,T_1u\rangle=c^2r^2.$$
    Let $(x_1,x_2,x_3)=(\Re\langle T_1u,u\rangle,\Im\langle T_1u,u\rangle,\langle T_1u,T_1u\rangle)\in DW(T_1)$, then we have,
    $$c^2(x_1^2+x_2^2)+x_3^2-2x_1x_3-(c^2-1)x_3=0.$$
    Hence the result follows.
\end{proof}

Now, we deal with composition operators with symbol $\phi(s)=c_0s+c_1$.

\begin{proposition}\label{isometry}
        Let $\phi(s)=c_0s+c_1$ be a $c_0$-symbol, where $c_0\in \mathbb{N}\cup\{0\}$, then the following statements hold.
        \begin{enumerate}[label=(\roman*)]
            \item If $c_0\geq 2$ and $c_1=ik$ where $k\in\mathbb{R}$, then $DW(C_\phi)=\{(\mu,1):\mu\in\mathbb{D}\}\cup\{(1,0,1)\}$.
            \item If $\Re c_1>0$, then 
            $DW(C_\phi)\subseteq\{(\lambda,r):|\lambda|^2\leq r\ \text{and}\ 0<r\leq1\}.$
        \end{enumerate}
\end{proposition}

\begin{proof}
    \begin{enumerate}[label=(\roman*)]
        \item From \cite[Theorem 17]{bayart2002hardy}, it follows that $C_\phi$ is an isometry and from \cite[Proposition 7]{finet2004numerical}, we have $W(C_\phi)=\mathbb{D}\cup\{1\}$. Hence the result follows.
        \item For $e_n(s)=n^{-s}$, we have $C_{\phi}e_n(s)=n^{-c_1}e_{n^{c_0}}(s)$. Therefore, for $\displaystyle f(s)=\sum_{n=1}^\infty a_n e_n(s)\in\mathcal{H}^2$ with $\displaystyle \|f\|=\sum_{n=1}^\infty|a_n|^2=1$, we obtain
        $\displaystyle \langle C_\phi f,f\rangle=\sum_{n=1}^\infty a_n \overline{a_{n^{c_0}}}n^{-c_1}$ and $\displaystyle \langle C_\phi f, C_\phi f\rangle=\sum_{n=1}^\infty |a_n|^2n^{-2\Re c_1}$. Thus,
        $$DW(C_\phi)=\left\{\left(\sum_{n=1}^\infty a_n \overline{a_{n^{c_0}}}n^{-c_1},\sum_{n=1}^\infty |a_n|^2n^{-2\Re c_1}\right):a_n\in\mathbb{C} \ \text{and}\ \sum_{n=1}^\infty |a_n|^2=1\right\}.$$
        The second co-ordinate is the convex combination of $\{n^{-2\Re c_1}:n\in\mathbb{N}\}$, i.e. $(0,1]$. From \cite[pg 386]{gau2021numerical}, we get the result.
    \end{enumerate}
\end{proof}

\begin{remark}\label{final}
    For normal composition operator $C_\phi$ defined on $\mathcal{H}^2$, by \cite[Theorem 15]{bayart2002hardy}, we have $\phi(s)=s+c_1$, where $\Re c_1\geq0$ and hence from \cite[Theorem 2.12, Corollary 2.15]{yao2021complex}, the spectrum of $C_\phi$ is $\{\overline{n^{-{c_1}}:n\in\mathbb{N}}\}$. Thus, \cite[Theorem 2.1]{li2009davis} implies
    $$DW(C_\phi)\subseteq\textbf{conv}\left\{(\lambda,|\lambda|^2):\lambda\in \{\overline{n^{-{c_1}}:n\in\mathbb{N}}\}\right\}=\overline{DW(C_\phi)}$$
    where $\textbf{conv}$ denotes the convex hull.
\end{remark}

We conclude the section with the following questions.

\begin{question}
    Let $\phi$ be of the form $\phi(s)=c_0s+c_1$ where $c_0\in \mathbb{N}\cup\{0\}$ and $\Re c_1\geq0$. Can the shape of $DW(C_\phi)$ be completely determined?
\end{question}

\begin{question}
    Suppose $\phi(s)$ is not of the form $\phi(s)=c_0s+c_1$ with $c_0\in \mathbb{N}\cup\{0\}$. What can be said about $DW(C_\phi)$?
\end{question}

\section*{Declarations}

\noindent \textit{Acknowledgements:} Mr. Subhadip Halder would like to thank the UGC, Govt. of India for the financial support (NTA Ref. No. 211610189204) in the form of fellowship.\\
\textit{Author Contributions:} All authors contributed equally to this work. \\
 \textit{Data Availability :} Not applicable.\\
\textit{Conflict of interest:} The authors declare that they have no competing interests.

\end{document}